\documentclass{amsart}
\usepackage{amsfonts}
\usepackage{hyperref}
\usepackage{graphicx}
\usepackage{tabularx}
\usepackage{array}
\usepackage[usenames,dvipsnames]{color}
\usepackage{comment}
\usepackage{amsmath}
\usepackage{amsthm}
\usepackage{amssymb}
\usepackage{fullpage}
\usepackage{xcolor}
\usepackage{tikz}
\usepackage{listings}

\tikzstyle{legend_general}=[rectangle, rounded corners, thin,
                          top color= white,bottom color=lavander!25,
                          minimum width=2.5cm, minimum height=0.8cm,
                          violet]

\DeclareMathOperator{\ch}{ch}
\renewcommand{\epsilon}{\varepsilon}

\newtheorem{theorem}{Theorem}[section]

\newtheorem{conjecture}[theorem]{Conjecture}

\newtheorem{lemma}[theorem]{Lemma}

\theoremstyle{definition}

\title{Bipartite graphs are $(\frac{4}{5}-\epsilon) \frac{\Delta}{\log \Delta}$-choosable}

\author{Peter Bradshaw}
\address{Department of Mathematics, University of Illinois Urbana-Champaign, Urbana, IL, USA}
\email{pb38@illinois.edu}

\author{Bojan Mohar}
\address{Department of Mathematics, Simon Fraser University, Burnaby, BC, Canada}
\email{mohar@sfu.ca}

\author{Ladislav Stacho}
\address{Department of Mathematics, Simon Fraser University, Burnaby, BC, Canada}
\email{lstacho@sfu.ca}

\thanks{Peter Bradshaw received support from NSF RTG grant DMS-1937241.}
\thanks{Research of Ladislav Stacho was supported by NSERC grant No. R611368.}
\thanks{
Bojan Mohar was supported in part by an NSERC Discovery Grant R832714 (Canada), and in part by the ERC Synergy grant KARST (European Union, ERC, KARST, project number 101071836).~
On leave from:
FMF, Department of Mathematics, University of Ljubljana, Ljubljana, Slovenia.}

\begin{document}

\begin{abstract}
Alon and Krivelevich conjectured that if $G$ is a bipartite graph of maximum degree $\Delta$, then the choosability (or list chromatic number) of $G$ satisfies $\ch(G)  = O \left ( \log \Delta \right )$. Currently, the best known upper bound for $\ch(G)$ is $(1 + o(1)) \frac{\Delta}{\log \Delta}$, which also holds for the much larger class of triangle-free graphs. We prove that for $\epsilon = 10^{-3}$, every bipartite graph $G$ of sufficiently large maximum degree $\Delta$ satisfies $\ch(G) < (\frac{4}{5}  -\epsilon) \frac{\Delta}{\log \Delta}$.
This improved upper bound suggests that list coloring is fundamentally different for bipartite graphs than for triangle-free graphs and hence gives a step toward solving the conjecture of Alon and Krivelevich.
\end{abstract}

\maketitle

\section{Introduction}

Given a graph $G$ for which each vertex $v \in V(G)$ has an associated list $L(v)$ of colors, an \emph{$L$-coloring} of $G$ is a proper coloring that assigns each vertex $v$ a color from $L(v)$. If $G$ has an $L$-coloring whenever $|L(v)| = k$ for each $v \in V(G)$, then $G$ is \emph{$k$-choosable}. We write $\ch(G)$ for the \emph{choosability} of $G$, defined as the least integer $k$ for which $G$ is $k$-choosable.
As $\chi(G) = k$ if and only if $G$ has an $L$-coloring for the assignment $L(v) = \{1,\dots, k\}$ to each vertex $v \in V(G)$,
it follows that $\chi(G) \leq \ch(G)$ for every graph $G$.

The choosability of a graph is often much larger than its chromatic number. As an example, Erd\H{o}s, Rubin, and Taylor proved that the choosability of the complete bipartite graph $K_{n,n}$ satisfies $ch(K_{n,n}) = (1+o(1)) \log_2 n$. More generally, Saxton and Thomason \cite{Saxton}  showed that if $G$ is a graph of minimum degree $\delta$, then $\ch(G) \geq \left ( 1 + o(1) \right ) \log_2 \delta$.
Alon and Krivelevich made the following conjecture in 1998, which states that this lower bound is best possible up to a constant factor when $G$ is bipartite:
\begin{conjecture}[\cite{AK}]
\label{conj}
If $G$ is a bipartite graph of maximum degree $\Delta$, then $\ch(G) = O ( \log \Delta)$.
\end{conjecture}
Conjecture \ref{conj} is still widely open, and currently the best known upper bound for the choosability of a bipartite graph of maximum degree $\Delta$ is $(1+o(1)) \frac{\Delta}{\log \Delta}$.
Molloy \cite{Molloy} proved this upper bound for the much larger class of triangle-free graphs, improving previous results of Johansson \cite{JohanssonTF}, Jamall \cite{Jamall}, and Pettie and Su \cite{PS} by a constant factor.
Bernshteyn \cite{BernshteynDP} 
used a simplified method to prove the same upper bound in the more general setting of correspondence colorings.
For triangle-free graphs, Molloy's upper bound is close to the best possible. Indeed, if $\Delta$ is fixed and $G$ is a large random $\Delta$-regular graph on $n$ vertices, then  asymptotically almost surely (a.a.s.), $\chi(G) =   (\frac{1}{2} + o(1)) \frac{\Delta}{\log \Delta}$ \cite{Frieze}, where the $o(1)$ term approaches $0$ as $\Delta$ increases. Furthermore, 
it is straightforward to show that the average degree of every small subgraph $H$ of such graph  $G$ (i.e.~$|H| < \log \log n$) is a.a.s.~less than $5$. Since the expected number of triangles in $G$ is less than $\Delta^3$ \cite{BB},
    it follows that removing all triangles from $G$ a.a.s.~reduces the chromatic number by at most $5$, giving a triangle-free subgraph $H$ of $G$ with maximum degree $\Delta$ satisfying $\ch(H) \geq \chi(H) \geq (\frac{1}{2} + o(1)) \frac{\Delta}{\log \Delta}$.

For triangle-free graphs, improving the coefficient in the upper bound $\ch(G) \leq (1+o(1)) \frac{\Delta}{\log \Delta}$ 
seems to be a difficult problem. Davies, de Joannis de Verclos, Kang, and Pirot \cite{DDKP} opine that reducing the $1+o(1)$ coefficient below $1$ would be a significant advance in current knowledge, and they point out that such an improvement would also improve a classical lower bound of Shearer \cite{Shearer} for the independence number of a triangle-free graph established in 1983. Furthermore, the upper bound $\ch(G) \leq (1+o(1)) \frac{\Delta}{\log \Delta}$
matches \emph{shattering threshold} for the problem of coloring random $\Delta$-regular graphs \cite{ZK}, also called \emph{algorithmic barrier} \cite{Ach}, a threshold that arises in many problems for random graphs.
For the problem of graph coloring, 
finding an efficient algorithm to color a random $\Delta$-regular graph
with $(1 - \epsilon) \frac{\Delta}{\log \Delta}$ colors 
is a major unsolved problem \cite{Ach, ZK}. Molloy \cite{Molloy} points out that since the triangles in a random $\Delta$-regular graph are few and sparsely distributed, an efficient algorithm that colors a $\Delta$-regular triangle-free graph with $(1 - \epsilon) \frac{\Delta}{\log \Delta}$ colors could also be applied to random $\Delta$-regular graphs and hence break the algorithmic barrier. 

Recently, Alon, Cambie, and Kang \cite{ACK} took a step toward answering Conjecture \ref{conj} using an approach related to the coupon collector problem. They showed that if $G$ is a bipartite graph of maximum degree $\Delta$ for which vertices in one partite set have color lists of size $(1 + o(1))\frac{\Delta}{\log \Delta}$, then the sizes of the color lists in the other partite set can be drastically reduced. In fact, their method implies that for any unbounded increasing function $\omega = \omega(\Delta)$, if the vertices in one partite set of $G$ have color lists of size $\omega$, then there exists a function $o(1)$ which approaches $0$ as $\Delta$ increases such that when each vertex in the other partite set has a list of size $(1 + o(1))\frac{\Delta}{\log \Delta}$, $G$ has a proper list coloring.

In this note, we prove that a bipartite graph $G$ admits a list coloring even when the list sizes in both parts of $G$ are reduced below $\frac{\Delta}{\log \Delta}$.
\begin{theorem}
\label{thm:main}
    If $G$ is a bipartite graph of sufficiently large maximum degree $\Delta$, then $\ch(G) < 0.797 \frac{ \Delta}{\log \Delta}$. 
\end{theorem}

Our proof uses the Lov\'asz Local Lemma and yields an efficient randomized algorithm via Moser's \emph{entropy compression} method \cite{Moser, Tao}.
Given the difficulty of
obtaining an upper bound of the form $(1-\epsilon) \frac{\Delta}{\log \Delta}$ for the choosability of a triangle-free graph of maximum degree $\Delta$, 
our result suggests that the list-coloring problem is 
fundamentally different for bipartite graphs than for triangle-free graphs, and that certain obstacles of the triangle-free setting do not appear in the bipartite setting. Hence, our result gives a step towards Conjecture \ref{conj}.
The main ingredient in our proof is the coupon collection argument used by Alon, Cambie, and Kang \cite{ACK}.
In particular, we show that when non-uniform probabilities are used in this coupon collector argument, then a similar argument yields an improved upper bound for the choosability of a bipartite graph.

\section{Main result}

In our proofs, we omit floors and ceilings, as they have little effect on our arguments.
We use the well-known Lov\'asz Local Lemma \cite{LLL}, stated in the following form \cite[Chapter 4]{MolloyReed}:

\begin{lemma}[Lov\'asz Local Lemma]
\label{lem:LLL}
Consider a set $\mathcal E$ of bad events such that for each $A \in \mathcal E$,
\begin{itemize}
    \item $\Pr(A) \leq p < 1$, and
    \item $A$ is mutually independent with all but at most $D$ of the other events.
\end{itemize}
If\/ $4Dp \leq 1$, then with positive probability none of the events in $\mathcal E$ occurs.
\end{lemma}

We also use the following well-known corollary of Jensen's inequality:
\begin{lemma}[{\cite[(3.6.1)]{Hardy}}]
\label{lem:Jensen}
    If $f:\mathbb R \rightarrow \mathbb R$ is a convex function and $x_1, \dots, x_t \in \mathbb R$, then
    \[f(x_1)+ \dots + f(x_t) \geq t f \left ( \frac{x_1+\dots+x_t}{t} \right ).\]
\end{lemma}

Before proving our main result, we need a lemma about the coupon collector problem, which takes place in the following setting. We let $L', L_1, \dots, L_{\Delta}$, be subsets of $\mathbb N$ of size exactly $k$. We often refer to the elements of $\mathbb N$ as \emph{colors}.
We let $0 < p < 1$ be some positive number (possibly dependent on $\Delta$). For each value $i$ ($1 \leq i \leq \Delta$), we define a probability distribution $P_i: L_i \rightarrow [0,p)$. Since $P_i$ is a probability distribution, we require that $\sum_{c \in L_i} P_i(c) = 1$, and for each color $c \not \in L_i$, we write $P_i(c) = 0$.
We also define independent random variables $\phi_1, \dots, \phi_{\Delta}$,
so that for each $i$ and $c \in L_i$, $\phi_i = c$ with probability $P_i(c)$.
For each $c \in L'$, we write $\rho(c) = \sum_{i = 1}^{\Delta} P_i(c)$.
Then, we prove the following lemma, which gives us an upper bound on the probability that for every color $c \in L'$, there exists some random variable $\phi_i$, so that $\phi_i=c$.
If each random variable $\phi_i$ represents the color of a coupon, 
then the probability that every color in $L'$ equals some random variable $\phi_i$ 
represents the probability that a coupon collector successfully collects a coupon in each color of $L'$ subject to our probability distributions.
The ideas in this lemma are similar to those of
Alon, Cambie, and Kang \cite[Section 3]{ACK}.

\begin{lemma}
\label{lem:coupon}
Let\/ $0 < \epsilon \leq 1$ and $0 < a  \leq 1$ 
be fixed, and let $\Delta$ be sufficiently large. 
Let $k = \left \lceil \frac{a \Delta}{(1 - p) (\log \Delta - 4 \log \log \Delta)} \right \rceil$.
Suppose that 
there exists a set $L^* \subseteq L'$ of size at least $\epsilon k$ such that the average value $\rho(c)$ for $c \in L^*$ satisfies 
\[\frac{1}{|L^*|}\sum_{c \in L^*} \rho(c) \leq \frac{ a \Delta }{ k }.\]
Then~ 
$\Pr \left (L' \subseteq \{\phi_1, \dots, \phi_{\Delta} \} \right ) < \exp(- \log^2 \Delta)$.
\end{lemma}
\begin{proof}
First, we show that 
\begin{equation}
\label{eqn:sumL}
 \Pr \left (L' \subseteq \{\phi_1, \dots, \phi_{\Delta} \} \right ) \leq \exp \left (  -\sum_{c \in L'} \exp \left (    - \frac{1}{1-p} \rho(c)    \right ) \right ).
\end{equation}
Consider a color $c \in L'$, and let $B_c$ be the event that $\phi_i = c$ for some value $i$ ($1 \leq i \leq \Delta$), i.e.~the event that the coupon collector obtains a coupon of color $c$. Since the variables $\phi_i$ are independent, $\Pr(B_c) = 1 - \prod_{i = 1}^{\Delta} (1 - P_i(c))$. Applying the inequality $1 - x \geq \exp (-{\frac{x}{1-x}} ) > \exp (-{\frac{x}{1-p}} )$ for $x < p$, we see that
\[\Pr(B_c) < 1 - \exp \left (-\frac{1}{1-p} \sum_{i = 1}^{\Delta} P_i(c) \right ) = 1 - \exp \left ( - \frac{1}{1 - p} \rho(c) \right  ).\]
Furthermore,
Alon, Cambie, and Kang~\cite[Section 3]{ACK}
show that the individual coupon collection events $\{B_c: c \in L'\}$ are negatively correlated, 
so the probability of the event 
$\bigcap_{ c\in L'} B_c$, or equivalently the event $L' \subseteq \{\phi_1, \dots, \phi_{\Delta} \}$, is less than
\[ \prod_{c\in L'} \left (1 - \exp \left( - \frac{1}{1 - p} \rho(c) \right ) \right ) \leq \exp \left ( - \sum_{c \in L'} \exp \left ( - \frac{1}{1 - p} \rho(c) \right  ) \right ),\]
proving (\ref{eqn:sumL}).

By possibly taking a subset of $L^*$, we assume without loss of generality that 
$|L^*| = \epsilon k$. 
By (\ref{eqn:sumL}), 
\begin{equation}
    \label{eqn:split}
    \Pr \left ( L' \subseteq \{\phi_1, \dots, \phi_{\Delta} \}  \right ) \leq   \exp \left ( - \sum_{c \in L'} \exp \left( - \frac{1}{1 - p} \rho(c) \right )  \right ) \leq 
 \exp \left ( - \sum_{c \in L^*} \exp \left( - \frac{1}{1 - p} \rho(c) \right )  \right ).
\end{equation}
Since the function $f(x) = e^{-x}$ is convex, and since $\frac{1}{\epsilon k}\sum_{c \in L^*} \rho(c) \leq \frac{ a \Delta }{k} $,
Lemma \ref{lem:Jensen} implies that
\[\sum_{c \in L^*} \exp \left ( - \frac{1}{1 - p} \rho(c) \right ) \geq \epsilon k \exp \left  (- \frac{a\Delta/k}{(1-p)} \right ). \]
Therefore, the argument of the outer exponential in (\ref{eqn:split}) is at most
\[  - \epsilon k \exp \left (- \frac{a \Delta / k }{(1-p)} \right ).\]
Now, if we substitute our value of $k$, then the argument of the outer exponential function in (\ref{eqn:split}) is at most $-  \Omega  \left ( \frac{ \Delta}{ \log \Delta} \right ) \exp (4 \log \log \Delta - \log \Delta ) < - \log^2 \Delta$,  for large enough $\Delta$, 
so the lemma holds.
\end{proof}

Before proving Theorem \ref{thm:main},
we prove the theorem with a weaker  coefficient of $\frac{4}{5} + o(1)$ as a warmup.
\begin{theorem}
\label{thm:45}
If $G$ is a bipartite graph of maximum degree $\Delta$, then $\ch(G) \leq  (\frac{4}{5} + o(1)) \frac{ \Delta}{\log \Delta}$.
\end{theorem}
\begin{proof}
We fix an arbitrarily small value $\gamma > 0$ and assume that the maximum degree $\Delta$ of $G$ is sufficiently large with respect to $\gamma$.
Without loss of generality, we assume that $G$ is $\Delta$-regular.
We let each vertex $v \in V(G)$ have a list $L(v)$ of $k = \left \lceil \frac{(4/5 + \gamma) \Delta}{(1 - 1/\sqrt{\Delta})(\log \Delta - 4 \log \log \Delta) } \right \rceil$ colors, represented as integers in increasing order.
We  show that 
$G$ has an $L$-coloring.

We partition $V(G)$ into two partite sets $A$ and $B$. We will create a probability distribution on each list $L(v)$ for $v \in A$ and  use these distributions to color the vertices $v \in A$ independently. Then, we will use Lemma \ref{lem:coupon} and the Lov\'asz Local Lemma to show that with positive probability, each vertex $w \in B$ still has an available color even after all vertices in $A$ have been colored.

For each vertex $v \in V(G)$, 
we write $L(v) = (c_1, \dots, c_{k})$ as an increasing integer sequence, and for each color $c \in L(v)$, we write $I(v,c) = i$ if and only if $c = c_i$---that is, if and only if $c$ is in the $i$th position in $L(v)$. We say that $I(v,c)$ is the \emph{index} of $c$ in $L(v)$.
For each vertex $w \in B$, and neighbor $v \in N(w)$, we define $\ell_{v,w} = |L(v) \cap L(w)|$. Then, 
for each vertex $w \in B$, we define the \emph{weight} of $w$ as 
\[Z(w) = \sum_{v \in N(w)} \ell_{v,w}.\]
 Clearly, for each vertex $w \in B$, $Z(w) \leq \Delta k$. 
%
For each vertex $v \in A$ and $c \in L(v)$, we write 
\[P_v(c) =  \frac{8/5}{k(1-\frac{3}{5k})} \left (1 - \frac{3}{4} \cdot \frac{ I(v,c)}{ k} \right ).\]
and for $c \in \mathbb N \setminus L(v)$, we write $P_v(c) = 0$.
For each color $c \in L(w)$, we write $N_c(w)$ for the set of neighbors $v \in N(w)$ satisfying $c \in L(v)$.
Observe that $\sum_{c \in L(v)}P_v(c) = 1$.
For each $w \in B$ and $c \in L(w)$, we write $\rho_w(c) = \sum_{v \in N_c(w)} P_v(c)$.

For each $v \in A$,
we color $v$ with a single color of $L(v)$ using the probability distribution $P_v$, so that $v$ receives each color $c \in L(v)$ with the probability $P_v(c)$.
Then, we  use the Lov\'asz Local Lemma to show that with a positive probability, our random coloring of $A$ can be extended to an $L$-coloring of $G$.
Observe that each color in $L(v)$ is used with a probability of (much) less than $1/\sqrt{\Delta}$.

Now,
we fix a vertex $w \in B$, and we aim to show that with probability at least $1 - \exp(-\log^2 \Delta)$, 
$L(w)$ contains a color which is not used 
to color any neighbor of $w$, so that we can extend our $L$-coloring of $A$ to $w$.
We write $z = \frac{Z(w) }{ \Delta k}$, and we fix a small constant $\epsilon > 0$. For each
color $c \in L(w)$ satisfying $I(w,c) \geq (1-\epsilon)k$, 
it holds for each $v \in N_c(w)$ that 
at most $\epsilon k$ colors $c' \in L(v) \cap L(w)$ satisfying $c' > c$ appear in $L(v)$; hence, 
$I(v,c) \geq \ell_{v,w}- \epsilon k$. Therefore,
\[\sum_{v \in N(w) \setminus N_c(w)} \frac 43 k + \sum_{v \in N_c(w)} I(v,c) \geq (z - \epsilon)\Delta  k,\]
as the term corresponding to $v \in N(w)$ contributes at least $\ell_{v,w} - \epsilon k$ to the sum. 
Therefore, for each color $c \in L(w)$ satisfying $I(w,c) \geq (1-\epsilon)k$,  
\begin{eqnarray*}
\rho_w(c) = \sum_{v \in N_c(w)} P_v(c) &=& \frac{8/5+o(1)}{k} \sum_{v \in N_c(w)   } \left (1 - \frac{3}{4} \cdot \frac{I(v,c)}{k} \right ) \\
 &=& \frac{8/5+o(1)}{k} \left ( \sum_{v \in N(w) \setminus N_c(w) } \left (1 - \frac 3{4k} \cdot  \frac 43 k \right ) +  \sum_{v \in N_c(w)   } \left (1 - \frac{3}{4k} \cdot I(v,c)\right ) \right )\\
  &=& \frac{8/5+o(1)}{k} \left ( \Delta - \frac {3}{4k} \left (\sum_{v \in N(w) \setminus N_c(w) } \frac 43 k +  \sum_{v \in N_c(w)   } I(v,c)\right ) \right ) \\
&\leq &  \frac{8/5+o(1)}{k}  \left (1 - \frac{3}{4}(z-\epsilon) \right ) \Delta. 
\end{eqnarray*}
Hence, for the last $\epsilon k$ colors $c \in L(w)$ (i.e.~those of largest index), the average value of $\rho_w(c)$ is at most $\left (\frac{8}{5}+o(1) \right ) \left ( 1 - \frac{3}{4} z + \frac{3}{4} \epsilon \right ) \frac{\Delta}{k}$.

On the other hand, 
by applying Lemma \ref{lem:Jensen} to the convex function $h(x) = \binom{x+1}{2} =  \frac{1}{2}x(x+1)$,
\[   \sum_{c \in L(w)} \sum_{ v \in N_c(w)   } I(v,c) 
\geq \sum_{v \in N(w)} \sum_{i=1}^{\ell_{v,w}} i = \sum_{v \in N(w)} \binom{\ell_{v,w} + 1}{2}
\geq \Delta \binom{Z(w) / \Delta + 1}{2} > \frac{1}{2}\Delta(zk)^2.
\]
Therefore, the average value $\rho_w(c)$ over all colors $c \in L(w)$ satisfies
\begin{eqnarray*}
 \frac{1}{k}\sum_{c \in L(w)} \rho_w(c) = \frac 1k \sum_{c \in L(w)} \sum_{v \in N_c(w)} P_v(c) & =& \frac{8/5+o(1)}{k^2} \sum_{c \in L(w)} \sum_{v \in N_c(w)}  \left (1 - \frac{3}{4} \cdot \frac{I(v,c)}{k} \right ) \\
 &<& \frac{8/5+o(1)}{k^2} \left ( Z(w) - \frac{3}{8} z^2 k \Delta \right ) \\
 &=& \left (\frac{8}{5}+o(1) \right )z \left (1-\frac{3}{8}z \right ) \frac{\Delta}{k}.  
\end{eqnarray*}

Hence, there exists
a subset $L^*(w) \subseteq L(w)$ of size at least $\epsilon k$ for which the average value $\rho_w(c)$ for $c \in L^*(w)$ is at most $\min \left  \{ \left ( 1 - \frac{3}{4} z + \frac{3}{4} \epsilon \right ) ,  z\left (1-\frac{3}{8} z \right ) \right \} \cdot \left ( \frac{8}{5} +o(1) \right ) \frac{\Delta}{k} < \left (\frac{4}{5} + \gamma \right ) \frac{\Delta}{k}$, where the inequality holds whenever $\epsilon$ is sufficiently small and $\Delta$ is sufficiently large with respect to $\gamma$. 

Now, for each vertex $w \in B$, we define a bad event $B_w$ as the event that after $A$ is randomly colored, no color in $L(w)$ is available---that is, that every color in $L(w)$ is used to color some vertex of $N(w)$. By applying Lemma \ref{lem:coupon} with our value of $\epsilon$, as well as with $a = \frac{4}{5} + \gamma$, $L' = L(w)$, $L^* = L^*(w)$, and $\{L_1, \dots, L_{\Delta}\} = \{L(v): v \in N(w)\}$, we find that $\Pr(B_w) < \exp(-\log^2 \Delta)$. 
Since each bad event occurs with probability less than $\exp(-\log^2 \Delta)$ and is independent with all but fewer than $\Delta^2$ other bad events, it follows from the Lov\'asz Local Lemma (Lemma \ref{lem:LLL}) that with a positive probability, no bad event occurs provided that $\Delta$ is large enough so that $4 \Delta^2 \exp(- \log^2 \Delta) \leq 1$. 
As we avoid all bad events $B_w$ with positive probability, it thus holds with positive probability that we can extend our $L$-coloring of $A$ to all of $G$.
Therefore, $G$ is $L$-colorable, and
the proof is complete.
\end{proof}
Next, we  show that the $\frac{4}{5} + o(1)$ coefficient from Theorem \ref{thm:45} can be reduced to $0.797$ using a coupon collection argument similar to that of Theorem \ref{thm:45}. While this improvement is minimal, 
the fact that the $\frac{4}{5} + o(1)$ coefficient can be broken with a similar argument
suggests that perhaps a more involved application of similar ideas can reduce the coefficient even more.

Before we prove that this lower coefficient can be achieved, we summarize the method used in Theorem \ref{thm:45}
and observe which parts of the method give room for improvement.
In our proof of Theorem \ref{thm:45}, we consider a vertex $w \in B$, and we hope to show that after randomly coloring all vertices in $A$, the probability that $w$ has no available color in $L(w)$ 
is small.
In order to show this, we aim to show that for some dense set of colors $c \in L(w)$, the values $\rho_w(c)$ are small. 
We write $Z(w) = z \Delta k$ for the weight of $w$, and we roughly describe two cases.

In the first case, if $z$ is large, then the colors $c \in L(w)$ appear at the lists $L(v)$ for neighbors $v \in N(w)$ with high frequency. Consequently,
the colors $c$ of large index $I(w,c)$ also have fairly large indices $I(v,c)$ for many neighbors $v \in N_c(w)$.
Since the probability of $c$
being used to color $v$ becomes small when $I(v,c)$ is large,
this means that colors $c \in L(w)$ of large index have small values $\rho_w(c)$. Specifically, we see in the proof of Theorem \ref{thm:45} that these colors $c$ of large index $I(w,c)$ approximately satisfy $\rho_w(c) \leq \frac{8}{5}(1 - \frac{3}{4} z) \frac{\Delta}{k}$.

In the second case, if $z$ is small, then for each neighbor $v \in N(w)$, $L(v)$ on average does not contain many colors from $L(w)$. Therefore, 
the average value $\rho_w(c)$ for all colors $c \in L(w)$ is small. Specifically, we see in the proof of Theorem \ref{thm:45} that the average value $\rho_w(c)$ is at most roughly $\frac{8}{5}z(1-\frac{3}{8}z) \frac{\Delta}{k}$. 

In both cases, we can find a dense set of colors $c \in L(w)$ for which the average value $\rho_w(c)$ is at most $(\frac{4}{5} + o(1)) \frac{\Delta}{k}$, with the upper bound being achieved when $z$ is close to $\frac{2}{3}$.
Now, let us consider 
the extremal case when this value $(\frac{4}{5} + o(1)) \frac{\Delta}{k}$ is achieved in more detail. 
When we compute
the upper bound $\frac{8}{5}z(1-\frac{3}{8}z) \frac{\Delta}{k}$
for the average value $\rho_w(c)$ over all colors $c \in L(w)$,
equality roughly holds only when the values $|L(v) \cap L(w)|$ are similar for each neighbor $v \in N(w)$ and when
the indices $I(v,c)$ for $c \in L(w)$ and $v \in N_c(w)$ are as low as possible.
Therefore, in the extremal case, for each neighbor $ v \in N(w)$, $|L(v) \cap L(w)| \approx \frac 23 k$, and the colors of $L(v) \cap L(w)$ roughly occupy
the first $\frac{2}{3}k$ indices of $L(v)$.
However, in this case, we can slightly increase the probabilities $P_v(c')$ for colors $c' \in L(v)$ with indices $I(v,c')$ close to $k$
without increasing the probabilities $P_v(c)$ of colors $c \in L(v) \cap L(w)$, as colors $c' \in L(v)$ with large indices $I(v,c')$ do not belong to $L(w)$.
This allows us to decrease the probabilities $P_v(c)$ of the colors $c \in L(v) \cap L(w)$, which reduces 
$P_v(c)$ for colors $c \in L(v) \cap L(w)$ and allows us us to reduce our coefficient below $\frac{4}{5}$.
On the other hand, if increasing the probabilities $P_v(c')$ for colors $c' \in L(v)$ of large index causes the probabilities $P_v(c)$ of many colors in $c \in L(v) \cap L(w)$ to increase, then this implies that the colors in $L(v) \cap L(w)$ for neighbors $v \in N(w)$ are not arranged as in the extremal case described above, and the method of Theorem \ref{thm:45} should still give a coefficient lower than $\frac{4}{5}$.

Using the approach outlined above,
we are ready to prove our improved coefficient.

\begin{theorem}
If $G$ is a bipartite graph of sufficiently large maximum degree $\Delta$, then $\ch(G) < 0.797 \frac{ \Delta}{\log \Delta}$.
\end{theorem}
\begin{proof}
We assume that the maximum degree $\Delta$ of $G$ is sufficiently large.
Without loss of generality, we
assume that $G$ is $\Delta$-regular.
We let each vertex $v \in V(G)$ have a list $L(v)$ of $k = 10 \left \lceil \frac{1}{10} \cdot \frac{0.7969 \Delta}{(1 - 1/\sqrt{\Delta})(\log \Delta - 4 \log \log \Delta) } \right \rceil$ colors, represented as integers in increasing order. We observe that $10$ divides $k$.
We aim to show that 
$G$ has a proper $L$-coloring.

We partition $V(G)$ into two partite sets $A$ and $B$. 
Again,
for each vertex $v \in V(G)$, 
we write $L(v) = (c_1, \dots, c_{k})$ as an increasing integer sequence, and for each color $c \in L(v)$, we write $I(v,c) = i$ if and only if $c = c_i$. We again say that $I(v,c)$ is the \emph{index} of $c$ in $L(v)$.
For each vertex $w \in B$ and neighbor $v \in N(v)$, we define again $\ell_{v,w} = |L(v) \cap L(w)|$. Then, for each $w \in B$, we again define the \emph{weight} of $w$ as 
\[Z(w) = \sum_{v \in N(w)} \ell_{v,w}.\]

We define a function $f:[1,k] \rightarrow \mathbb R$ as follows:
\[
f(x) =
\begin{cases}
1 -  \frac{3}{4 k}x  & \textrm{ if } x \leq \frac{9}{10} k , \\
\frac{13}{40} & \textrm{ if }  \frac{9}{10}k < x \leq k.
\end{cases}
\]
We write $C$ for the average value of $f(i)$ over $i \in \{1, \dots, k\}$ and observe that $C = \frac{503}{800} + o(1)$.
For each $c \in L(v)$, we write 
\[P_v(c) = \frac{1}{C k}  f \left ( I(v,c) \right ).\]
If $c \not \in L(v)$, we write $P_v(c) = 0$.
Observe that $\sum_{c \in L(v)}P_v(c) = 1$. 
For each $w \in B$ and $c \in L(w)$, we again write $\rho_w(c) = \sum_{v \in N_c(w)} P_v(c)$, where $N_c(w)$ is the set of neighbors $v \in N(w)$ for which $c \in L(v)$.
As before, for each $v \in A$,
we color $v$ with a single color of $L(v)$ using the probability distribution $P_v$, so that $v$ receives each color $c \in L(v)$ with probability $P_v(c)$.
We aim to show that with positive probability, we can extend our $L$-coloring of $A$ to an $L$-coloring of $G$.

We fix a vertex $w \in B$,
and we aim to show that with probability at least $1 - \exp(-\log^2 \Delta)$,
some color of $L(w)$ is not used to color any neighbor of $w$, so that our $L$-coloring of $A$ can be extended to $w$.
As before, we write $z = \frac{Z(w) }{ \Delta k}$.
We define $0 \leq y \leq 1$ so that exactly $y \Delta$ neighbors $v \in N(w)$ satisfy $\ell_{v,w}> \frac{9}{10}k$.
We write $N'(w)$ for the set of $y \Delta$ neighbors $v \in N(w)$ for which $ \ell_{v,w}> \frac{9}{10}k$, and we write $N''(w) = N(w) \setminus N'(w)$ for the remaining set of $(1-y)\Delta$ neighbors of $w$.
We observe that
\[ z  \geq \frac{1}{k\Delta} \sum_{v \in N'(w)} \ell_{v,w} > \frac{9}{10} y.\]
We define $\alpha$ so that $\sum_{v \in N'(w)} \left ( \ell_{v,w} - \frac{9}{10}k \right ) = \alpha yk\Delta$, and we also observe that
$0 \leq \alpha  \leq \frac{1}{10}$.
As $ \sum_{v \in N(w)} \ell_{v,w} = zk\Delta $ and $\sum_{v \in N'(w)} \ell_{v,w} = (\alpha y + \frac 9{10}y )k\Delta$, it follows that 
\begin{equation}
\label{eqn:N''}
\sum_{v \in N''(w)} \ell_{v,w} = \sum_{v \in N(w)} \ell_{v,w} - \sum_{v \in N'(w)} \ell_{v,w} = \left (z - y \left (\alpha + \frac {9}{10} \right ) \right ) k\Delta.
\end{equation}

Now, 
we fix a small constant $\epsilon > 0$, and assume that $\Delta$ is sufficiently large with respect to $\epsilon$. We consider a color $c \in L(w)$ for which $I(w,c) \geq (1-\epsilon) k$.
As before, for each $v \in N_c(w)$, $I(v,c) \geq \ell_{v,w} - \epsilon k$.
We compute an upper bound on $\rho_w(c)$ as follows, using the fact that $f$ is decreasing 
 and is
$\frac{3}{4k}$-Lipschitz.
\begin{eqnarray*}
\rho_w(c) &=& \frac{1}{Ck} \sum_{v \in N_c(w)} f(I(v,c)) \leq  \frac{1}{Ck} \sum_{v \in N(w)} f(\ell_{v,w} - \epsilon k)  \\
&\leq & \frac{3 \epsilon}{4Ck} \Delta  + \frac{1}{Ck} \sum_{v \in N(w)} f(\ell_{v,w})  \\
&= & \frac{3 \epsilon}{4Ck} \Delta  + \frac{1}{Ck} \left ( \sum_{v \in N''(w)} f(\ell_{v,w}) + \sum_{v \in N'(w)} f(\ell_{v,w})  \right )\\
&= & \frac{3 \epsilon}{4Ck} \Delta  + \frac{1}{Ck}\left ( \sum_{v \in N''(w)} \left ( 1 - \frac{3 }{4k}  \ell_{v,w} \right ) + \sum_{v \in N'(w)}  \left ( 1 - \frac{3}{4k}  \ell_{v,w} + \frac{3}{4k} \ell_{v,w} - \frac{27}{40} \right ) \right )\\
 &=& \frac{3 \epsilon}{4Ck} \Delta  + \frac{1}{Ck} \left ( \sum_{v \in N(w)} \left ( 1 - \frac{3}{4k}  \ell_{v,w} \right ) + \sum_{v \in N'(w)} \left (\frac{3}{4k} \ell_{v,w} - \frac{27}{40} \right ) \right )\\
 &=&  \frac{3 \epsilon}{4Ck} \Delta  + \frac{1}{Ck}  \left ( \Delta \left ( 1 - \frac{3}{4} z \right ) + \frac{3}{4k} \sum_{v \in N'(w)} \left ( \ell_{v,w} - \frac{9}{10}k \right )  \right )\\
 &=& \frac{\Delta}{C k} \left (  1 - \frac{3}{4} z  + \frac{3}{4} \alpha y + \frac{3}{4}\epsilon \right ).
\end{eqnarray*}

Hence, the average value of $\rho_w(c)$
for the $\epsilon k$ colors $c \in L(w)$
with greatest indices $I(w,c)$
is at most $\frac{\Delta}{Ck}\left ( 1 + \frac{3}{4} (-z +  \alpha y + \epsilon ) \right ) $.

On the other hand, the average value $\rho_w(c)$ over all colors $c \in L(w)$ satisfies
\begin{eqnarray}\
\notag
 \frac{1}{k}\sum_{c \in L(w)} \rho_w(c) &=&  \frac{1}{C k^2} 
 \sum_{v \in N(w)} 
 \sum_{c \in L(v)}  f(I(v,c)) 
 \leq
 \frac{1}{C k^2} 
 \sum_{v \in N(w)}
 \sum_{i=1}^{\ell_{v,w}}
 f(i) \\
 \notag
 & = & 
 \frac{1}{C k^2} \left (
 \sum_{v \in N''(w)}
 \sum_{i=1}^{\ell_{v,w}} \left (1- \frac{3}{4k}i \right ) + \sum_{v \in N'(w)} \left (
 \sum_{i=1}^{\frac{9}{10}k }
 (1- \frac{3}{4k}i )  + \sum_{i = \frac{9}{10}k + 1}^{\ell_{v,w}} \frac{13}{40} \right ) \right ) \\
 &< & 
 \notag
 \frac{1}{C k^2} 
 \sum_{v \in N''(w)}
 \sum_{i=1}^{\ell_{v,w}} \left (1- \frac{3}{4k}i  \right )+ \frac{1}{C k^2}  \sum_{v \in N'(w)} \left (
 \frac{9}{10}k- \frac{3}{4k} \cdot \frac{1}{2} \left (\frac{9}{10}k \right )^2 + \frac{13}{40} \left (\ell_{v,w} - \frac{9}{10} \right )
 \right )
 \\
  \label{eqn:gross}
  &=&
  \frac{1}{C k^2} 
 \sum_{v \in N''(w)} \left (
\ell_{v,w} - \frac{3}{4k} \binom{\ell_{v,w} + 1}{2}  \right )
 +  \frac{y\Delta}{Ck} \left (   \frac{477 }{800 }    +
  \frac{13 }{40}  \alpha  \right ).
\end{eqnarray}

By (\ref{eqn:N''}), the average value $\ell_{v,w}$
for $v \in N''(w)$
is $ \overline{\ell} := \frac{k(z - (\frac{9}{10} + \alpha)y)}{1-y}$.
With this notation, we have 
$\sum_{v \in N''(w)} \ell_{v,w} = (1-y) \Delta \overline{\ell}$.
Furthermore, by applying Lemma \ref{lem:Jensen} to the convex function $h(x) = \binom{x+1}{2} =  \frac{1}{2}x(x+1)$,
\[\sum_{v \in N''(w)} \binom{\ell_{v,w} + 1}{2} \geq (1-y)\Delta \binom{ \overline{\ell}  + 1 }{2}  > (1-y) \Delta \cdot \frac 12 \overline{\ell}^2. \]
Therefore, (\ref{eqn:gross}) implies that
\begin{eqnarray*}
 \frac{1}{k}\sum_{c \in L(w)} \rho_w(c)
&<& 
\frac{(1-y)\Delta}{Ck^2} \left ( \overline{\ell} - \frac{3}{4k} \cdot \frac{1}{2}\overline{\ell}^2\right )
 +  \frac{y\Delta}{Ck} \left (   \frac{477 }{800 }    +
  \frac{13 }{40}  \alpha  \right ) \\
 &=& 
\frac{\Delta}{Ck} \left [ \left  (z - \left (\frac{9}{10} + \alpha \right )y \right ) \left ( 1 - \frac{3}{8} \cdot  \frac{z - \left (\frac{9}{10} + \alpha \right )y}{1-y}\right )
+  y \left (   \frac{477 }{800 }    +
  \frac{13 }{40}  \alpha  \right ) \right ] .
 \end{eqnarray*}
Hence, writing $g(\alpha,y,z) = z - \left (\frac{9}{10} + \alpha \right )y$, there exists a dense subset $L^*(w) \subseteq L(w)$ of size at least $\epsilon k$ for which the average value $\rho_w(c)$ for $c \in L^*(w)$ is at most \[ \frac{\Delta}{Ck} \min \left  \{ 1 + \frac{3}{4} (-z +  \alpha y + \epsilon )   ,   g(\alpha,y,z) \left ( 1 - \frac{3}{8} \cdot \frac{g(\alpha,y,z)}{1-y}\right )
+  y \left (   \frac{477 }{800 }    +
  \frac{13 }{40}  \alpha  \right )  \right \}.\]

We would like to show that this quantity is less than $\frac{0.7969 \Delta}{k}$ when $\epsilon$ is sufficiently small and $\Delta$ is sufficiently large. To establish this upper bound, we first observe that if $z - \alpha y > 0.66535$, then $\frac{\Delta}{Ck} ( 1 + \frac{3}{4} (-z +  \alpha y + \epsilon ) ) < (0.7968  + \frac{3}{4} \epsilon + o(1) ) \frac{\Delta}{k}$, which is smaller than $\frac{ 0. 7969 \Delta}{k}$ when $\epsilon$ is sufficiently small and $\Delta$ is sufficiently large. Hence, we may assume that $z - \alpha y \leq 0.66535$. Since $y \leq 1$ and $\alpha \leq 0.1$, this implies in particular that $z < 0.8$. Furthermore, since $z > 0.9y$, we may therefore assume that $y < 0.9$.
We would like to show that under these constraints, 
\[\frac{\Delta}{Ck} \left ( g(\alpha,y,z) \left ( 1 - \frac{3}{8} \cdot \frac{g(\alpha,y,z)}{1-y}\right )
+  y \left (   \frac{477 }{800 }    +
  \frac{13 }{40}  \alpha  \right )   \right ) <  \frac{ 0. 7969 \Delta}{k},\]
which will prove our upper bound. To this end, we execute the following commands in Maple:
\begin{verbatim}
    h := (a, y, z) -> 800/503*(z - (0.9 + a)*y)* 
    (1 + (-1)*0.375*(z - (0.9 + a)*y)/(1 - y)) + 800/503*y*(477/800 + 13/40*a):
    
    with(Optimization):
    
    Maximize(h(a, y, z), {0 <= a, a <= 0.1, 0 <= y, y <= 0.9, 0 <= z, z <= 0.8,
    -a*y + z <= 0.66535});
\end{verbatim}
This gives us the following output:
\begin{verbatim}
     [0.796309237086130106, [a = 0.100000000000000, y = 0.202933582180192, 
     z = 0.685643358218019]]
\end{verbatim}
As a result, we find that under our constraints on $\alpha$, $y$, and $z$, our expression is less than $(0.7964 + o(1)) \frac{\Delta}{k}$, which is less than our desired upper bound when $\Delta$ is sufficiently large. Hence, there exists a dense subset $L^*(w) \subseteq L(w)$ of at least $\epsilon k$ colors for which the average value $\rho_w(c)$ for $c \in L^*(w)$ is less than $ \frac{0.7969 \Delta}{k}$.

As before, for each vertex $w \in B$, we define a bad event $B_w$ to be the event that all colors of $L(w)$ are used by the neighbors of $w$. By applying Lemma \ref{lem:coupon} with our value $\epsilon$, as well as with $a = 0.7969$, $L = L(w)$, $L^* = L^*(w)$, and $\{L_1, \dots, L_k\} = \{L(v): v \in N(w)\}$, we find that $\Pr(B_w) < \exp(-\log^2 \Delta)$. As before, we apply the Lov\'asz Local Lemma (Lemma \ref{lem:LLL}) when $\Delta$ is sufficiently large to find that with positive probability, no bad event occurs. Hence,
with positive probability, our random $L$-coloring of $A$ extends to an $L$-coloring of $G$, completing the proof.
\end{proof}

The coefficient of $0.797$ in Theorem \ref{thm:main} is not the best possible, and small improvements can be made through slight adjustments to the function $f$.
However, making significant additional improvements to the coefficient using this method seems difficult without additional ideas.

\section{Acknowledgement}
We are grateful to Abhishek Dhawan for helpful discussions about an earlier draft of this paper.

\bibliographystyle{abbrv}
\bibliography{bib}

\end{document}